\documentclass[11pt]{amsart}
\usepackage{palatino}
\usepackage{amsfonts}
\usepackage{amssymb}
\usepackage{amscd}

\usepackage{amsmath}
\usepackage{amsfonts}
\usepackage{amsthm}
\usepackage{amssymb}
\usepackage{amscd}
\usepackage[all]{xy}
\usepackage{enumerate}



\theoremstyle{plain}

\newtheorem{thm}{Theorem}[section]

\newtheorem{prop}[thm]{Proposition}

\newtheorem{cor}[thm]{Corollary}
\newtheorem{lem}[thm]{Lemma}

\theoremstyle{definition}
\newtheorem{defn}[thm]{Definition}

\newtheorem*{ackn}{Acknowledgment}

\theoremstyle{remark}
\newtheorem*{rem}{Remark}
\newcommand{\sB}{\mathcal{B}}

\newcommand{\sE}{\mathcal{E}}

\newcommand{\sH}{\mathcal{H}}

\newcommand{\sM}{\mathcal{M}}
\newcommand{\sO}{\mathcal{O}}

\newcommand{\sU}{\mathcal{U}}

\newcommand{\mP}{\mathbb{P}}

\numberwithin{equation}{section}

\newcommand{\beba}  {\begin{equation}\begin{array}{rcl}}

\newcommand{\eaee}  {\end{array}\end{equation}}

\newcommand{\mg} {{\overline {\mathcal  M}}_{g}}

\title{ A note on Harris Morrison sweeping families of maximal gonality}

\author{ Beorchia Valentina and Zucconi Francesco}

\address{Dipartimento di Matematica e Geoscienze,
               Universit\`a degli studi di Trieste,\\
via Valerio 12/b, 34127 Trieste, Italy\\
\texttt{beorchia@units.it}}

\address{Dipartimento di Matematica e Informatica\\
via delle Scienze 206,
Universit\`a degli studi di Udine\\
Udine, 33100 Italy\\
\texttt{Francesco.Zucconi@uniud.it}}

\begin{document}

\maketitle

\markboth{Beorchia and Zucconi}{Harris-Morrison sweeping families}

\begin{abstract} In \cite[theorem 2.5]{HMo} Harris and Morrison construct
   semistable families $f\colon F\to Y$ of $k$-gonal curves of genus $g$ 
   such that for 
   every $k$ the corresponding modular curves give a sweeping family 
   in the $k$-gonal 
   locus $\overline{\sM^{k}_{g}}$. Their construction depends on the choice of
   a smooth curve $X$. We 
   show that if the genus $g(X)$ is sufficiently high with respect to 
   $g$ then the ratio $\frac{K^{2}_{F}}{\chi(\sO_{F})}$ is $8$ 
   asymptotically with respect to $g(X)$.
   Moreover, if the conjectured estimates given in 
   \cite[p. 351-352]{HMo} hold, we show that if $g$ is big enough, then
    $F$ is a surface of positive index.
 \end{abstract}
\tableofcontents

\section{Introduction.}
Let $\overline\sM_{g}$ be the 
moduli space of stable curves of genus $g$. 
A family $\sB$ of curves $Z\subset {\overline {\mathcal M}_g}$
is said to be a sweeping family if 
$\cup\sB$ is a Zariski dense subset of $\overline\sM_{g}$. The 
invariants 
of sweeping families can furnish important information on the geometry of 
$\overline\sM_{g}$; see: \cite{CFM}.

In this paper we focus on the geometry of the 
families of curves constructed in \cite[Theorem 2.5]{HMo}. We revise this 
construction in Section 2; here we point out only that 
it depends on choosing a smooth curve $X$ of genus $g(X)$. 
For the aim of \cite{HMo} $g(X)$ can be taken to be zero, but Harris 
Morrison construction applies then to any curve $X$.

We compute the invariants of the smooth surface $F$ supporting the Harris Morrison family. We denote by 
$e(F)$, $\chi(\sO_{F})$, $K^{2}_{F}$ respectively the topological characteristic of $F$, its 
cohomological characteristic and the self-intersection 
of the canonical divisor. We also denote by $q(F)$ the irregularity 
of $F$. 
 We 
have:

\medskip

\noindent
{\bf{Theorem}}{\it{ Let $g,k\in\mathbb N$ such 
       that $k$ can occur as the gonality of a smooth curve of genus 
       $g$. There exist Harris Morrison sweeping 
       $k$-gonal semistable 
       genus-$g$ fibrations $f \colon F\to Y$ 
       such that the ratio 
       $\frac{K^{2}_{F}}{\chi(\sO_{F})}$
       is $8$ asymptotically with respect to $g(X)>0$.}}
\medskip

 \noindent
 See Theorem \ref{secondaadifferenza} for a refined statement. We 
 stress that our Theorem is not obtained by a base change over a 
 family as in \cite[Theorem 2.5]{HMo} starting from $\mP^{1}$. We 
 remark that 
 the above result applies to any gonality $k$.

 If we consider only those 
 families with maximal gonality, then we show that (see Proposition
 \ref{secondo}): 

\medskip

\noindent
{\bf Proposition} {\it{Let $f\colon F\to Y$ be an 
    Harris Morrison genus-$g$ fibration 
   starting from any plane curve $X$ of 
       genus $g(X)>>0$. Assume 
    that the gonality $k$ is maximal and that the conjectured estimates 
    in \cite[p. 351-352]{HMo} are
    true. If $g$ is big enough, then $F$ is a surface of 
    positive index i.e. $K^{2}_{F}> 8\chi(\sO_{F})$. Moreover 
    if $f\colon F\to Y$
    is general in its class then the irregularity of $F$ is
    $g(Y)$. In particular, $f\colon F\to Y$ is the Albanese morphism 
    of $F$.}}
\medskip

 We recall that any surface of general type satisfies the Miyaoka-Yau 
 inequality $K^{2}_{F}\leq 9\chi(\sO_{F})$, and equality holds if and 
 only if $F$ is a ball quotient. Surfaces of positive index satisfy the inequalities 
 $8\chi(\sO_{F})< K^{2}_{F}{\bf < }\
 9\chi(\sO_{F})$. These surfaces are still quite mysterious objects 
\cite{Re}, \cite{MT}, \cite{My}. 

The statement of the Proposition above is close to the recent
results of Urz{\'u}a.
He has found in  \cite{Ur1}
simply connected projective surfaces of general type with 
$\frac{K_{F}^{2}}{e(F)}$ arbitrarily close to $\frac{71}{26}$; this  
improves previous estimates by Persson-Peters-Xiao \cite{PPX}. By his 
method, based on a generalisation of the method of line arrangements on 
$\mP^{2}$, he finds interesting surfaces of positive 
index; see: \cite{Ur2}. Thanks to the Kapranov construction of the moduli space
${\overline{M}}_{0,d+1}$ of rational curves with $d+1$ marked points, 
see \cite{Ka1}, \cite{Ka2}, line arrangements correspond to
curves inside ${\overline{M}}_{0,d+1}$. In this sense Urz{\'u}a 
construction is related to the Harris Morrison one. Hence the fact that 
our results on Harris 
Morrison construction of $k$-gonal sweeping families match those 
obtained in \cite{Ur2} is an evidence for the conjectural results 
stated in \cite{HM}.

\ackn The authors would like to thank D. Chen 
 for his comments on the first version of the paper.
 
 This research is supported by MIUR funds, 
PRIN project {\it Geometria delle variet\`a algebriche} (2010), 
coordinator A. Verra. The first author is also supported by funds of the
Universit\`a degli Studi di Trieste - Finanziamento di Ateneo per
progetti di ricerca scientifica - FRA 2011.

\bigskip\bigskip

\section{ The basic construction.}

\subsection{Harris-Morrison families.}

We review and we explain some features of the basic construction of \cite{HMo}.
We could construct a slight refinement of 
the basic construction of \cite[Section 2]{HMo} using not a product 
surface $X\times\mP^{1}$ but a ruled surface $\mP(\sE)\to X$ over 
$X$, but we 
have decided to follow \cite[Section 2]{HMo} for simplicity.

We denote by $[N]$ and $N$ a finite set and its 
order respectively. If 
$[M]$ is a subset of $[N]$ and is invariant under the action of 
$\mathfrak{S}_k$ we will denote the quotient $[M]/\mathfrak{S}_k$ by 
$[\widetilde M]$.

Let ${\overline{M_{0, b}}}$ be the compactification of the moduli space of $b$-pointed stable 
curves of genus $0$ and let ${\overline\beta}\colon{\overline{H_{k, 
b}}}\to {\overline{M_{0, b}}}$ be the natural morphism on the
Hurwitz scheme of admissible 
$k$-covers of stable $b$-pointed curves of genus $0$ constructed in 
\cite{HM}. Then the morphism $\beta\colon H_{k,b}\to\sM_{0,b}$ is a 
$\widetilde N(k,b)$ sheeted unramified covering, where $\widetilde N(k,b)$ counts the 
$k$-sheeted {\it{connected}} covers of $\mP^{1}$ simply branched at 
$b$ fixed points. From now on we set $N:=N(k,b)$, where  $N(k,b)$ is 
defined in \cite[page 334]{HMo}. Concerning the 
subset $[N]$ we recall that $\mathfrak{S}_k$ acts on $[N]$ and 
by \cite[Lemma 1.24]{HMo} this action is free if 
$k\geq 3$ and trivial for $k=2$, so $N=k!\ \widetilde N$ if $k\geq 3$ and 
$N=\widetilde N$ if $k=2$.

Let $X$ be any smooth complete curve
and let $\pi_{X}\colon X\times\mP^{1}\to X$ be the projection.
By \cite{Ha} ${\rm{NS}}(X\times\mP^{1})=[s]\mathbb Z\oplus [f]\mathbb 
Z$ where $[s]$ is the numerical class of a $\pi_{X}$-section $s$
and $[f]$ is the numerical class of a 
$\pi_{X}$-fiber $f$. In the sequel we will not distinguish between 
$s$ and its class $[s]$ if no danger of confusion can arise.

Let $s$ be a $\pi_{X}$-section with $s^{2}=0$. Let $b\in \mathbb N$ and
let $C_{1},\ldots, C_{b}$ be effective divisors on $X$ of degrees $c_i >0$,
such that $\sO_{X\times\mP^{1}}(s+\pi_X ^\star C_i)$ is very ample, $i=1,\ldots 
,b$. Let 
$\sigma_{i}\in |s+\pi_X ^\star C_i|$ be a smooth curve for 
every $i=1,\ldots, b$,  such that the sections $\sigma_{i}$ 
meet transversely everywhere and such that 
each $\pi_{X}$-fiber contains at most one of these intersections. 

We set 
$$
\sigma:=\sum \sigma_{i}
$$ and we point out that $[\sigma]= b[s]+[\pi_X ^\star C]=b[s]+c[f]$ 
where $C=\sum C_{i}$ and $c=\sum c_{i}$.
The reader is warned that if $X=\mP^{1}$ the above costruction easily 
works. If $X$ is a curve of positive genus then it is sufficient 
to assume that for every $i=1,\ldots, b$ the number $c_{i}$ is 
at least equal to the degree of a very ample divisor on $X$.

Let $[I_{X}]$ 
be the set of nodes of $\sigma$ or equivalently the set of 
intersections of the $\sigma_{i}$'s. Then 
$$
I_{X}=(b-1)c;
$$
\noindent
see: \cite[Formula 2.1 on page 338]{HMo}.

Let $B_{X}$ be the blow-up of 
$X\times\mP^{1}$ at $[I_{X}]$ and denote by $B_{x}$ its fiber over $x\in X$ 
and $\widetilde \sigma_{i}$ the strict transform of $\sigma_{i}$ on 
$B_{X}$. Let $\alpha\colon X\to{\overline{\sM_{0,b}}}$ be the map 
which sends a point $x\in X$ to the class of the fiber $B_{x}$ marked 
by its $b$ points of intersections with the $\widetilde\sigma_{i}$ 
where $i=1,\ldots, b$. 
Let $Y:=X\times_{{\overline{\sM_{0,b}}}}{\overline{H_{k, 
b}}}$. By the same argument of \cite[pages 338-339]{HMo} the induced 
map $\mu\colon Y\to X$ is a covering of degree $\widetilde N$. 
Let $\pi_{Y}\colon Y\times\mP^{1}\to Y$ be natural projection.
 If we set $\nu:=\mu\times{\rm{id}}\colon Y\times\mP^{1}\to 
X\times\mP^{1}$ we see that 
$[I_{Y}]:=\nu^{\star}([I_{X}])$ is the scheme of singularities of the 
divisor  
$$
\tau:=\nu^{\star}\sigma
$$ on $Y\times\mP^{1}$. Let $[J_{Y}]\subset
Y$ be the $\pi_{Y}$-projection of $[I_{Y}]$ and let $[J_{X}]\subset 
X$ be the $\pi_{X}$-projection of $[I_{X}]$. By construction we have 
a morphism from the complement of $[J_{Y}]$ to ${\overline{\sM_{g}}}$, 
and we extend it to a morphism $\rho\colon Y\to {\overline{\sM_{g}}}$ 
following the argument of 
\cite[Theorem 2.15]{HMo}.

Here we recall first that 
$\mu\colon Y\setminus [J_{Y}]\to X\setminus [J_{X}]$ is unramified and 
that $[J_{Y}]$ is partitioned 
   into three classes: $$[J_{Y}]=[J_{Y, (1)}]\cup [J_{Y, (2,2)}]\cup 
   [J_{Y, (3)}].$$\noindent
  The topological description given in
   \cite[Page 340]{HMo} yields that $\mu$ is not ramified on 
   $[J_{Y, (1)}]\cup [J_{Y, (2,2)}]$ and it is triply branched at 
   points of $[J_{Y, (3)}]$. In particular 
   $$
   I_{Y}=\left({\widetilde{N_{1}}}+{\widetilde{N_{2,2}}}+
   \frac{{\widetilde{N_{3}}}}{3}\right)
   (b-1)c;$$
 where ${\widetilde N}={\widetilde{N_{1}}} 
 +{\widetilde{N_{ 2,2}}}+{\widetilde{N_{3}}}$   
 and a combinatorical definition of 
   ${\widetilde{N_{1}}}, {\widetilde{N_{2,2}}}, {\widetilde{N_{3}}}$ 
   is given in \cite[Proposition 2.9]{HMo}.

   Following 
\cite{HMo} we blow-up $\epsilon\colon A\to Y\times\mP^{1}$ at the points of 
the set $[I_{Y,(1)}]$ 
and  we construct a finite cover $\pi\colon G\to A$ and a semistable 
fibration $f\colon F\to Y$
   see \cite[Diagram (2.2), page 338]{HMo} such that over the complement 
   $Y'$ of $[J_{Y}]$ in $Y$ the surface $G$ is obtainable
   by the pull-back of the 
   smooth universal admissible cover 
   $\theta\colon\sU_{\sH_{k,b}}\subset\mP^{1}\times 
   \sH_{k,b}\to\sH_{k,b}$ whose existence is shown in \cite{HM}. We 
   recall that over $Y'$, $G$ is a stable fibration $G_{Y'}\to 
   Y'\subset Y$, it coincides 
   with $f\colon F\setminus f^{-1}([J_{Y}])=F_{Y'}\to Y'$. Moreover 
   $A$ coincides with $Y\times\mP^{1}$ over $Y'$ and there 
   exists a {\it{finite}} morphism $G_{Y'}=F_{Y'}\to 
   A_{Y'}=\mP(\sO_{Y'}\oplus\sO_{Y'})$.\

   By the very explicit 
   local analytic description of the {\it{finite}} cover $\pi\colon G\to A$ and of the 
   blow-up $\epsilon\colon A\to Y\times\mP^{1}$  done in 
   \cite[from page 343 to the top of page 347]{HMo} we obtain an 
   explicit description of the smooth 
   surface $F$ and of the semistable fibration $f\colon 
   F\to Y$ inducing the desired moduli map 
   $\rho\colon Y\to{\overline{\sM_{g}}}$. We set $Z:=\rho(Y)$. We sum 
   up the above construction in the following diagram:
   
   \bigskip
   
   \begin{equation}\label{diag}
\xymatrix{
 S \ar[dr]^h \ar@{<-}[d]^u &  &&\\
G  \ar[r]^\zeta \ar[d]_\pi^{k:1}  & F\ar[dr]^f \ar@{-->}[d]&& \qquad Z \subset \mg \ar@{<-}[dl]^\rho\\
 A \ar[r]^{\epsilon } & Y\times\mP^{1}\ar[r]^{\pi_Y} \ar[d] _\nu^{\widetilde N:1} & Y \ar[d]_\mu ^{\widetilde N:1}\ar[r]& \overline{H}_{k,b}\ar[d]^{\overline\beta}\\
  &  X\times\mP^{1} \ar[r]^{\pi_X} & X \ar[r]^\alpha &\overline {\sM_{0,b}}\\
}
\end{equation}

\bigskip

   What makes this families very interesting is that thanks to the 
   work done in \cite{HMo} the dominant morphism $\zeta\colon G\to F$ 
   is very explicit and hence
the nature of the fibers of $f\colon F\to Y$ 
over $[J_{Y}]$ is evident.

If $y\in Y$ we denote by 
	$F_{y}$ and $G_{y}$ respectively the $f$-fiber over $y$ and the 
	$(f\circ \zeta)$-fiber over $y$.

	The analysis of fibers of type $(1)$, that is $y\in [J_{Y, 
	(1)}]$, splits into four cases: 
	$(1_{0})$, $(1_{j, 0})$
	$(1_{j, g})$, $(1_{j, i})$ where $j>0$ and $1\leq i\leq 
	g/2$; see \cite[Diagram (2.12)]{HMo} and with obvious 
	notation we write $y\in [J_{Y,\Pi}]$ if  $y\in [J_{Y, 
	(1)}]$ and $y$ is of type $\Pi$.

	If $y\in [J_{Y, (1_{0})}]$ then $F_{y}=F'_{y}\cup E$ where 
	$F'_{y}$ is a genus $g-1$ curve and
	$E$ is a rational $(-2)$ curve; 
	see the top of diagram \cite[Diagram (2.12)]{HMo}. The 
	corresponding fiber $G_{y}$ on $G$ is 
	$$
	G_{y}=G''_{y}\cup E'\cup\bigcup_{j=1}^{k-2}E_{j,y}
	$$ 
	where $G''_{y}$ is a 
	genus $g-1$ curve,
	$E_{j,y}^{2}=-1$ and $E'$ is a -$2$ curve. Over a 
	neighbouhood of $y$ the morphism $\zeta\colon G\to F$ 
	consists on the contraction of the $k-2$ rational curves 
	$E_{j,y}$.

	If $y\in [J_{Y, (1_{j, 0})}]$ or $y\in [J_{Y, 
	(1_{j, g})}]$ then $F_{y}$ is a {\it{smooth}} semistable 
	fiber and over a 
	neighbouhood of $y$ the morphism $\zeta\colon G\to F$ 
	consists on a {\it{ordered contraction}} of $k$ rational 
	curves. More precisely if $y\in [J_{Y, (1_{j, 0})}]\cup [J_{Y, 
	(1_{j, g})}]$ then 
	$$
	G_{y}=G'_{y}\cup E'_{y}\cup 
	G''_{y}\cup \bigcup_{i=1}^{k-2}E_{i,y}
	$$ 
	where 
	$G'_{y}$ is a rational curve, $E'_{y}$ is a $(-2)$-rational 
	curve, $E'_{y}\cdot G'_{y}=1$, $G''_{y}$ is a smooth curve of genus $g$,
	$E'_{y}\cdot G''_{y}=1$,
	$G''_{y}$ is a smooth curve of genus $g$ and $E_{i,y}^{2}=-1$.  
	The map $\zeta$ first contracts 
	$E_{i,y}$ where $i=1,\ldots ,k-2$. After this contraction process
	the image of $G'_{y}$ is a $-1$ curve which intersects the 
	image of $E'_{y}$ in a unique point. Hence to get the semistable reduction $F$
	we need {\it{first}} to contract the image of $G'_{y}$ and then the 
	image of $E'_{y}$.

	If 
	$y\in [J_{Y, (1_{j, i})}]$ where $j>0$ and $1\leq i\leq [g/2]$ 
	then $F_{y}=F_{y,i}\cup E_{y}\cup F_{y,g-i}$ where 
	$F_{y,i}\cap F_{y,g-i}=\emptyset$, $E_{y}^{2}=-2$, 
	$F_{y,i}\cdot E_{y}= F_{y,g-i}\cdot E_{y}=1$ and 
	$$
	G_{y}= 
	G'_{y}\cup E'_{y}\cup G''_{y}\cup\bigcup_{i=1}^{k-2}E_{i,y},
	$$
	where $E_{i,y}$ are $-1$-rational curves, $E_{y}$ is a $(-2)$ curve. 
	Moreover $F_{y,i}$ is smooth of genus $i$ and admits a $j$ 
	covering over the fiber $\mP^{1}_{y}\subset Y\times\mP^{1}$ and 
	$F_{y,g-i}$ is smooth of genus $g-i$ and admits a $k-j$ 
	covering over $\mP^{1}_{y}$. Note that if $e_{y}\subset A$ is 
	the exceptional curve over $y$ arised in the blow-up 
	$A\to Y\times\mP^{1}$ then the restriction of $\pi\colon G\to A$ 
	to $E'_{y}\cup\left(\bigcup_{i=1}^{k-2}E_{i,y}\right)$ gives the $k$-cover of 
	$e_{y}$ contained in $G$. Over a 
	neighbouhood of $y$ the morphism $\zeta\colon G\to F$ 
	consists on the contraction of the $k-2$ rational curves 
	$E_{i,y}$.

	\begin{lem}\label{tipodueduetre}
    If $y\in [J_{Y, (2,2)}]\cup [J_{Y, (3)}]$ then the fiber $F_{y}$ 
    of $f\colon F\to Y$ over $y$ is a smooth semistable curve. 
    The rational map $F\dashrightarrow Y\times\mP^{1}$ is finite in a 
    neighbourhood of $F_{y}$. Moreover the surfaces $G$ and 
   $F$ coincide over an analytic neighboorhood of $y\in Y$.
    \end{lem}
    \begin{proof} See \cite[Pages 343-346]{HMo}.
	\end{proof}

\begin{prop}
    \label{daGaS} 
    Over each of the points of $[I_{X}]=\pi_{X}([J_{X}])\subset X$ there are ${\widetilde{N}}_{{\rm{sing}}}$ 
	    points $y\in Y$ such that the fiber $F_{y}\subset F$ is 
	    singular. Any singular fiber of $f\colon 
	    F\to Y$ contains a unique $(-2)$ rational curve. 

\end{prop}
	\begin{proof} This is a restatement of \cite[Theorem 
	    3.1]{HMo} and it easily follows by the above description 
	    of the singularities of the fibration $f\colon F\to Y$. 
	\end{proof}

\begin{defn} A semistable fibration $f\colon F\to Y$ as the one of 
    Proposition \ref{daGaS} is called a Harris-Morrison family.
\end{defn}

Note that by Harris Morrison construction it follows
\begin{prop}\label{harris} Let $X$ be any smooth complete curve of genus $g(X)$. 
    Let $g\geq 3$ be any natural number.
    If $g$ is odd let $k=\frac{g+3}{2}$ and set 
    $k=\frac{g+2}{2}$ if $g$ is even. Then the family of curves 
    $Z\subset {\overline {\mathcal M}_g}$ constructed by Harris and 
    Morrison varying the curve $X$ forms a sweeping family.
\end{prop}
\begin{proof} The chosen number $k$ is the maximal gonality for a curve of 
    genus $g$. Hence the claim follows by a standard result of 
    Brill-Noether theory. See also \cite[page 350]{HMo}. 
\end{proof}

\begin{rem} Notice that, even without the assumption of maximal 
    gonality, the 
    families of curves $Z\subset {\overline {\mathcal M}_g}$ 
    constructed by Harris and 
    Morrison varying the curve $X$ depend freely on $g(X)$. In 
    particular we will study those families such that the  
    parameter $g(X)>>0$.
\end{rem}
	
\subsection{Numerical invariants.}

We recall 
that $\epsilon\colon A\to Y\times\mP^{1}$ is the blow-up of $[I_{Y, (1)}]$ 
which is a reduced scheme of length 
$(b-1)c{\widetilde{N_{1}}}$. Let $E_{A}$ be the exceptional 
divisor of $\epsilon\colon A\to Y\times\mP^{1}$. Hence

\begin{equation}\label{autointe}
    E_{A}^{2}=-I_{Y, (1)}=-(b-1)c{\widetilde{N_{1}}}.
  \end{equation}
  Let 
$\widetilde\tau$ be the strict transform of $\tau$, that is 
$\widetilde\tau=\epsilon_{A}^{\star}(\tau)-2E_{A}$. Since 
$\tau=\nu^{\star}(\sigma)=\nu^{\star}(bs+cf)$ then

\begin{equation}\label{tautildequadro}
    (\widetilde\tau)^{2}=2bc{\widetilde{N}}-4(b-1)c{\widetilde{N_{1}}}.
    \end{equation}

Let $R$ be the ramification divisor of $\pi\colon G\to A$. 
The following 
relation holds: 

\begin{equation}\label{pistarR}
    \pi_{\star}R=\widetilde\tau.
    \end{equation}
    
    Since we have simple 
branch on the fibers the local analysis shows that 
\begin{equation}\label{pistabassopistaralto}
    \pi^{\star}\pi_{\star}R=2R+\widetilde R
\end{equation}

and that:

\begin{equation}\label{erreerretilde}
R\cdot\widetilde R=(b-1)c\left(2{\widetilde{N}}_{2,2}+\frac{4}{3}{\widetilde{N}}_{3}\right).
\end{equation}

\begin{prop}\label{erreerre} The ramification divisor $R$ of the finite 
    cover $\pi\colon G\to A$ satisfies:
    
    $$R^{2}=\left({\widetilde N}+(b-1)\left(\frac{1}{3}{\widetilde N_{3}}-
    {\widetilde N_{1}}\right)\right)c
    $$  
\end{prop}
 \begin{proof} It is as in
    \cite[Theorem 2.15]{HMo}.
\end{proof}

By construction we have seen that over any point $y\in [J_{Y,(1)}]$ 
there exist $k-2$ rational curves $E_{i,y}\subset G_{y}$ such that 
$E_{i,y}^{2}=-1$. In particular there are 
$(k-2){\widetilde{N_{1}}}(b-1)c$ 
exceptional curves contained inside fibers of $f_{G}\colon G\to Y$ 
over $J_{Y,(1)}$ .

\begin{lem}\label{laesse} Let $u\colon G\to S$ be the contraction of these 
$(k-2){\widetilde{N_{1}}}(b-1)c$ rational curves. Then $S$ is a smooth surface. 
Moreover $K_{S}^{2}=K_{G}^{2}+ (k-2){\widetilde{N_{1}}}(b-1)c$.
\end{lem} 
\begin{proof} It follows straightly from Castelnuovo contraction 
theorem.
\end{proof}

Now we want to compute the invariants of $S$.

\begin{prop}\label{eccez=0} Let $E$ be any $-1$ rational curve which 
    is contracted by $u\colon G\to S$. Then $R\cdot E=0$.
   
 \end{prop}
 \begin{proof} We use the ramification divisor formula: 
     $K_{G|Y}=\pi^{\star}K_{A|Y}+R$. Let $E$ be any $-1$ rational curve which 
    is contracted by $u\colon G\to S$ then the finite morphism 
    $\pi\colon G\to A$ restricts to an isomorphism $\pi_{|E}\colon 
    E\to L\subset A$ where $L$ is an exceptional curve produced via 
    the blow-up $\epsilon\colon A\to Y\times\mP^{1}$. In particular 
    $\pi_{\star}E=L$. Then $R\cdot 
     E=-1-\pi^{\star}(K_{A|Y})E=-1-K_{A|Y}\cdot L=-1+1=0$.
 \end{proof}
 
 \begin{cor}\label{erretildeconmenouno} Let $E$ be any $-1$ rational curve which 
    is contracted by $u\colon G\to S$. Then $E\cdot{\widetilde R}=2$.
\end{cor}
\begin{proof}  By Equation \ref{pistabassopistaralto} and by 
    Proposition \ref{eccez=0} we have: 
    $0=E\cdot 
    R=E\cdot\frac{1}{2}(\pi^{\star}\pi_{\star}R-\widetilde R)=
    \frac{1}{2}(L\cdot\widetilde\tau-E\cdot\widetilde R)$. Since 
    $\widetilde\tau=\epsilon^{\star}(\tau)-2E_{A}$ and $E_{A}\cdot 
    L=-1$ then the claim follows.
    \end{proof}
 
    We now consider the $(-2)$ rational curves of $G$. We notice that if 
    $y\in[J_{Y,{1_{(2,0)}}}]\cup [J_{Y,{1_{(2k-2,g)}}}]$ then there 
    is a $(-2)$ 
    rational curve
    $G'_{y}\subset G_{y}$ which is mapped to a fiber of the canonical 
    projection $\pi_{Y}\colon Y\times\mP^{1}\to Y$; that is: 
    $\epsilon\circ\pi(G'_{y})=\pi_{Y}^{-1}(y)\subset Y\times\mP^{1}$. Next 
    lemma deals with the $(-2)$ curves which are mapped to the 
    exceptional curves of the morphism $\epsilon\colon A\to 
    Y\times\mP^{1}$.

 \begin{lem}\label{erretildeconmenodue} Let $E$ be any $(-2)$ rational curve which 
    is contained in a fiber of $f_{G}\colon G\to Y$ and which is 
    mapped $2$-to-$1$ to a cuve $L\subset A$ such that 
    $\epsilon(L)\in [I_{Y}]$. 
    Then $R\cdot E=2$. Moreover $E\cdot\widetilde R=0$
 \end{lem}
 \begin{proof} If $E$ is a $(-2)$ curve contained in a fiber of 
     $f_{G}\colon G\to Y$ which is mapped to an 
    exceptional curve $L$ of the morphism 
    $\epsilon\colon A\to Y\times\mP^{1}$ then $\pi_{|_E}\colon E\to L$ is a 
    $2$-to-$1$ covering. Hence $\pi_{\star}E=2L$. 
    Then $0=E\cdot K_{G|Y}=E\cdot 
    (\pi^{\star}K_{A|Y}+R)=2L^{2}+E\cdot R$. This implies $R\cdot 
    E=2$. Then 
    $E\cdot{\widetilde{R}}=E\cdot\pi^{\star}\pi_{\star}R-4=
    2L\cdot\widetilde\tau-4=0$.
    \end{proof}
 
 By the 
local description of $f\colon S\to Y$ we know that there are 
$(b-1)c\left[{\widetilde{N}}_{{\rm{sing}}}
+\sum_{j=0}^{[\frac{k}{2}]}({\widetilde{M_{j,0}}}+{\widetilde{M_{j,g}}})\right]$ mutually disjoint $(-2)$ 
rational curves, where the $[M_{j,i}]$ are given in \cite[Formula 
1.35]{HMo}.
\begin{lem}\label{combinatorics}
    $\sum_{j=0}^{[\frac{k}{2}]}({\widetilde{M_{j,0}}}
    +{\widetilde{M_{j,g}}})\leq {\widetilde{N_{1}}}$.
\end{lem}
\begin{proof} By \cite[Formula 1.30]{HMo} 
    $[N_{1}]=\cup_{i=0}^{[\frac{k}{2}]}[M_{j}]$ where $[M_{j}]$ is 
    defined in \cite[Formula 1.28]{HMo}. Since by \cite[Formula 
    1.36]{HMo} $[M_{j}]=\cup_{i=1}^{g}[M_{j,i}]$ the claim follows.
 \end{proof}

Moreover let 
\begin{equation}\label{discrepanza}
    e:=\sum_{j=0}^{[\frac{k}{2}]}(\widetilde{M_{j,0}}+{\widetilde{M_{j,g}}}).
    \end{equation}
    \noindent
    Obviously 
$(b-1)ce=\sum _{j=0}^{[\frac{k}{2}]}(J_{Y,1_{j,0}}+J_{Y,1_{j,g}})$. We set 
$$
\bigcup _{j=0}^{[\frac{k}{2}]}([J_{Y,1_{(j,0)}}]\cup [J_{Y,1_{(j,g)}}])=:\{a_{1},\ldots 
,a_{(b-1)ce}\}\subset Y.
$$
 We recall that the morphism $h\colon S\to F$ 
contracts two rational curves of any fiber of $f_{S}\colon S\to Y$ over $\{a_{1},\ldots 
,a_{(b-1)ce}\}\subset Y$. Moreover the fiber $F_{a_{i}}$ $i=1,\ldots 
,(b-1)ce$ of $f\colon F\to Y $is smooth, but over a neighbourhood $U_{i}$ of 
$a_{i}\in Y$ the map 
$\rho_{F}\colon F_{U}\dashrightarrow U\times\mP^{1}$ is {\it{only a rational 
one}}. We know that the fiber $f_{S}^{\star}(a_{i}):=S_{a_{i}}=
F'_{a_{i}}+E_{a_{i}}+E'_{a_{i}}\subset S$ where $F'_{a_{i}}$ is a smooth curve 
of genus $g$, $E_{a_{i}}$ is a $(-2)$ rational curve such that 
$E_{a_{i}}\cdot F'_{a_{i}}=1$. Moreover we know that 
$E'_{a_{i}}$ is a $-1$-rational curve and $E_{a_{i}}\cdot E'_{a_{i}}=1$.

\begin{cor}\label{calcolodiscrepanza}
    Let $h\colon S\to F$ be the morphism which factorises the 
    morphism $\zeta\colon G\to F$. Then $K_{F}^{2}=K_{S}^{2}+2(b-1)ce$.
\end{cor}
\begin{proof} By Harris Morrison construction the contraction of 
    $k-2$ rational curves in each one of the $(b-1)c{\widetilde N}_{1}$ 
    singular fibers of $f_{G}\colon G\to Y$ gives the morphism 
    $u\colon G\to S$. To reach $F$ we need to contract first 
    $E'_{a_{i}}$ and then the image of $E_{a_{i}}$ for every 
    $i=1,\ldots, (b-1)ce$. Hence $\zeta\colon G\to F$ is given by 
    $h\zeta=h\circ u$ where $h\colon S\to F$ is a composition of 
    $2(b-1)ce$ simple contractions. Hence $K_{F}^{2}=K_{S}^{2}+2(b-1)ce$.
    
\end{proof}

By Lemma \ref{tipodueduetre} 
the points of type $(2,2)$ and the points of type $(3)$ 
    give a smooth semistable fiber of $f\colon F\to Y$. By 
    Proposition \ref{daGaS} there are 
    ${\widetilde{N}}_{{\rm{sing}}}(b-1)c$ singular fibers of $f\colon F\to 
    Y$. Set 
    $$
    r:={\widetilde{N}}_{{\rm{sing}}}(b-1)c.
    $$
     Let 
	    $\{y_{i}\}_{i=1}^{r}={\rm{Sing}}(f)$ 
	    be the subset of $Y$ given by the points $y\in Y$ 
	    such that $f^{-1}(y)=F_{y}$ is a singular fiber. By the 
	    local description of the singular fibers of $f\colon F\to 
	    Y$ it follows that for any $l=1,\ldots, r$, 
	    $f^{-1}(y_{l})=F_{y_{l}}=F_{y_{l},i}\cup E_{y_{l}}\cup 
	    F_{y_{l},g-i}$ 
	    where $E_{y_{l}}$ is a $(-2)$ curve, 
	    $F_{y_{l},i}\cdot E_{y_{l}}=
	    F_{y_{l},g-i}\cdot E_{y_{l}}=1$ and obviously 
	    $F_{y_{l},i}\cdot
	    F_{y_{l},g-i}=0$, or $F_{y_{l}}=F'_{y_{l}}\cup E$ where 
	$F'_{y_{l}}$ is a genus $g-1$ curve and
	$E$ is a rational $(-2)$ curve.
\medskip
	
  We study the invariants of the surface $F$.  
\begin{prop}\label{eulerodif}Let $f\colon F\to Y$ be a Harris 
     Morrison family. Then the topological Euler characteristic of 
     $F$ is:
     $$
     e(F)=4(g-1)(g(Y)-1)+2r.
     $$
    
\end{prop}
\begin{proof}
    By a standard 
    topological argument it follows that 
    $e(F)=4(g-1)(g(Y)-1)+\sum_{i=1}^{r} (e(F_{y_{i}})-(2-2g))$. By the 
    analysis of singular fibers, see Proposition \ref{daGaS}, it 
    follows that for every $i=1,\ldots, r$
    $e(F_{y_{i}})-(2-2g)=2$. 
    Then $e(F)=4(g-1)(g(Y)-1)+2r$.
    \end{proof}

We point out the reader 
the very important relation:
\begin{equation}\label{tuttoli}
{\widetilde{N}}_{1}=e+{\widetilde{N}}_{{\rm{sing}}}
\end{equation}\noindent
where $e$ is defined in Equation \ref{discrepanza}.

\begin{prop}\label{kappaquadroeffe} 
    Let $f\colon F\to Y$ be an Harris Morrison 
    semistable fibration. Then
    $$
    K^{2}_{F}=c\left[(b-1)\left[2e 
    +{\widetilde{N}}_{1}+(8g-7)\frac{{\widetilde{N}}_{3}}{3}
    \right]-3{\widetilde{N}}\right] +8{\widetilde{N}}(g(X)-1)(g-1)
    $$
\end{prop}
\begin{proof} We notice that by Corollary \ref{calcolodiscrepanza} 
    $K^{2}_{F}=2e(b-1)c+K^{2}_{S}$. Then by Lemma \ref{laesse} 
    $K^{2}_{F}=2e(b-1)c+(k-2)(b-1)c{\widetilde{N}}_{1}+K^{2}_{G}$. By definition 
    of $R$ we have 
    $K_{G}^{2}=R^{2}+2R\cdot\pi^{\star}K_{A}+(\pi^{\star}K_{A})^{2}$. 
    By Equation \ref{autointe} we have 
   $$
   (\pi^{\star}K_{A})^{2}=kK^{2}_{A}=-k\left[8(g(Y)-1)+(b-1)c{\widetilde{N}}_{1}\right].
   $$
   \noindent
   By projection formula and by Equation \ref{pistarR} we have:
   $$
   2R\cdot\pi^{\star}K_{A}=2{\widetilde{\tau}}K_{A}=
   \left[ 2b(g(X)-1)-2c\right]{\widetilde{N}}+b(b-1)c\frac{2{\widetilde{N}}_{3}}{3}.
   $$
   \noindent Now the claim follows by Proposition \ref{erreerre} 
   taking into account that $b=2g+2k-2$ and 
    \begin{equation}\label{generedi Y}
     2g(Y)-2={\widetilde{N}}(2g(X)-2)+\frac{2}{3}(b-1)c{\widetilde{N}}_{3}.
     \end{equation}
     since Riemann-Hurwitz formula.
   
\end{proof}

We put 
\begin{equation}\label{alpha}
    \alpha:=(b-1)\left[2e 
    +{\widetilde{N}}_{1}+(8g-7)\frac{{\widetilde{N}}_{3}}{3}
    \right]-3{\widetilde{N}}
 \end{equation}
 \noindent
    
\begin{cor}\label{notdependence} For every 
    fibration constructed as in \cite[Theorem 2.5]{HMo} we can write by 
    Proposition \ref{kappaquadroeffe} and by Equation \ref{alpha}:
    \begin{equation}\label{facile}
    K^{2}_{F}=c\alpha+8{\widetilde{N}}(g(X)-1)(g-1)
    \end{equation}
    \noindent where the coefficient $\alpha$ does not depend on 
    $g(X)$.
 \end{cor}

\begin{prop}\label{kappaquadroeffehaarris} If $f\colon F\to Y$ is the 
    fibration constructed in \cite[Theorem 2.5]{HMo} then
   $$
   \chi(\sO_{F})=\frac{c}{12}\left[(b-1)\left[ 
    +3{\widetilde{N}}_{1}+(12g-11)\frac{{\widetilde{N}}_{3}}{3}
    \right]-3{\widetilde{N}}\right] 
    +{\widetilde{N}}(g(X)-1)(g-1) .
    $$
\end{prop}
 \begin{proof} By Noether Identity and by 
     Proposition \ref{eulerodif} we have 
     $$
     12\chi(\sO_{F})=K^{2}_{F}+4(g-1)(g(Y)-1) 
     +2{\widetilde{N}}_{{\rm{sing}}}(b-1)c.
     $$\noindent 
     By Proposition \ref{kappaquadroeffe} we can write
     $$
     12\chi(\sO_{F})=c\left[(b-1)\left(2e 
    +{\widetilde{N}}_{1}+(8g-7)\frac{{\widetilde{N}}_{3}}{3}
    \right)-3{\widetilde{N}}\right] +
    $$
    $$
    \qquad \qquad \qquad \qquad +8{\widetilde{N}}(g(X)-1)(g-1)+4(g-1)(g(Y)-1) 
     +2{\widetilde{N}}_{{\rm{sing}}}(b-1)c.
     $$
     \noindent By Equation \ref{tuttoli} and by Equation 
     \ref{generedi Y} we next obtain: 
$$
 12\chi(\sO_{F})=c\left[(b-1)\left( 
    3{\widetilde{N}}_{1}+(12g-11)\frac{{\widetilde{N}}_{3}}{3}
    \right)-3{\widetilde{N}}\right] 
    +12{\widetilde{N}}(g(X)-1)(g-1) 
    $$
   and the claim follows. 
 \end{proof}
 
 We put
 \begin{equation}\label{alphaprimo}
 \alpha':=\frac{1}{12}\left[(b-1)\left(
    3{\widetilde{N}}_{1}+(12g-11)\frac{{\widetilde{N}}_{3}}{3}
    \right)-3{\widetilde{N}}\right].
 \end{equation}\noindent
    In particular we have:
    
 \begin{cor}\label{notdependencebis} 
     For every 
    fibration constructed as in \cite[Theorem 2.5]{HMo} we can write 
    by Proposition \ref{kappaquadroeffehaarris} and by Equation 
    \ref{alphaprimo}:
    
\begin{equation}\label{facilefacile}
\chi(\sO_{F})=\alpha'c+{\widetilde{N}}	(g(X)-1)(g-1).
\end{equation}
\noindent where the coefficient $\alpha'$ does not depend on $g(X)$. 
\end{cor}

 \begin{cor}\label{primadifferenza} Let $X$ be any smooth complete 
     curve with genus $g(X)$. Let $f\colon F\to Y$ be the 
    fibration constructed in \cite[Theorem 2.5]{HMo} then
    $$
     K^{2}_{F}-8\chi(\sO_{F})=c\left[(b-1)
     \left[2e-{\widetilde{N}}_{1}+\frac{{\widetilde{N}}_{3}}{9}\right]
     -{\widetilde{N}}\right].
     $$
  \end{cor}
  \begin{proof}  It follows by Corollary \ref{kappaquadroeffehaarris} 
      and by Proposition \ref{kappaquadroeffe}.
      
   \end{proof}
   
\subsection{The proof of the Theorem} Now we want to analyse the 
relation between the fiber degrees $c_{i}$ for the divisors 
$\sigma_{i}\in |s+\pi_X ^\star C_i|$ and the genus $g(X)$ in 
order that a family $f\colon F\to Y$ as 
      in \cite[Theorem 2.5]{HMo} exists if we start from a curve $X$. 
      Since we are working over a product surface $X\times\mP^{1}$ a 
      divisor $L$ which is numerically equivalent to $s+c_{i}f$ is very ample 
      iff $L_{|(s=0)}$ is very ample; hence a necessary condition is 
      that $c_{i}$ is the degree of a very ample divisor 
      on $X$. On the other hand if $\sO_{X}(l)$ is a very ample sheaf 
      on $X$ then $s+\pi_{X}^{\star}(l)$ is a very ample divisor on 
      $X\times\mP^{1}$. 
      For a general $X$  if $\sO_{X}(l)$ is a nonspecial  very ample 
      divisor then by 
      Halphen theorem \cite[Proposition 6.1]{Ha} it follows that 
      $c_{i}\geq g(X)+3$. Then if we want to construct families 
      $f\colon F\to Y$ as
      in \cite[Theorem 2.5]{HMo} starting from a curve $X$ where $c$ 
      is small with respect to $g(X)$, definitely we need to 
      consider curves $X$ with very ample special divisors.

   \begin{thm}\label{secondaadifferenza} Let $g,k\in\mathbb N$ with 
   $3\le k \le \lfloor {(g+3)\over 2} \rfloor$. For every real number 
   $\epsilon >0$, there exists  
       a real number $\Delta(\epsilon)\geq 0$ 
       such that there are families $f\colon F\to Y$ obtained 
       by the Harris Morrison construction 
       starting from any plane curve $X$ of 
       genus $g(X)\geq \Delta(\epsilon)$ such that the following holds:
    $$
    8-\epsilon\leq \frac{K^{2}_{F}}{\chi(\sO_{F})}\leq 
       8+\epsilon.
       $$
     \end{thm}
     \begin{proof} Since $X$ is a plane curve of genus $g(X)$ then by 
	 Clebsh formula its degree is 
	 $d(X)=\frac{3+\sqrt{8g(X)+1}}{2}$. We can consider 
	 $c_{i}=d(X)$ where $i=1,\ldots, b$ where $b=2(g+k-1)$. Then 
	 $c$ can be taken equal to $bd(X)$. For every Harris-Morrison family we can write:
      
\begin{equation}\label{facilefacilefacile}
	  \frac{K^{2}_{F}}{\chi(\sO_{F})}=
	  \frac{c\alpha+8{\widetilde{N}}(g(X)-1)(g-1)}{c\alpha'+{\widetilde{N}}
	  (g(X)-1)(g-1)}.
\end{equation}
Since the parameters $\alpha$, 
$\alpha'$ {\it{do not depend}} on $g(X)$ since $g,k$ are fixed and 
since $d(X)=\frac{3+\sqrt{8g(X)+1}}{2}$ then we can find 
$\frac{2g(X)-2}{3+\sqrt{8g(X)+1}}\geq  
\frac{b||\alpha-8\alpha'|-\epsilon\alpha'|}{\epsilon(g-1)}$ to obtain 
$|\frac{K^{2}_{F}}{\chi(\sO_{F})}-8|\le \epsilon$.
\end{proof}

\begin{rem}\label{pianodifferenza}
    Theorem \ref{secondaadifferenza} can be extended easily to subcanonical 
    curves $X$ where the subcanonical degree is sufficiently small with respect 
to $g(X)$.
\end{rem}

\begin{rem} We point out the reader that
    the invariants of 
    surfaces of general type 
    which supports 
    fibrations as those of Proposition 
    \ref{pianodifferenza} are strongly influenced by the base $Y$ of 
    the fibration, in a way which is quite new for the theory of 
    surfaces of general type, as far as we know.  
\end{rem}

\noindent We have shown Theorem stated in the Introduction.

 \section{Maximal gonality and surfaces of positive index.}

In \cite{HMo} the genus $g(X)$ plays no role, 
see: \cite[Corollary 3.15]{HMo}. In this work it plays 
an essential role due to the Equations \ref{facile} and Equation \ref{facilefacile}.

We consider the expressions of $K^{2}_{F}$ and $\chi(\sO_{F})$ as 
(linear) polynomials in the variable $g(X)$. In this section we consider 
the Harris Morrison families obtained in \cite{HM} and with 
{\it{maximal}} gonality.

In particular in this section we have:

\begin{equation}\label{casodisparo}
    g=2n+1, k=n+2, b=6n+4, m=(6n+3)c
 \end{equation}
 \noindent or
 
 \begin{equation}\label{casoparo}
    g=2n, k=n+1, b=6n, m=(6n-1)c
 \end{equation}
 \bigskip
 
\noindent
Let us see how $k$ influences $K^{2}_{F}$ and $8\chi(\sO_{F})$. 
We will use the following 
\begin{lem}\label{tecnico} Let $f\colon F\to Y$ be an 
    Harris Morrison genus-$g$ fibration as in \cite[Theorem 2.5]{HMo}. Assume 
    that the gonality $k$ is maximal, that is $k=\frac{g+3}{2}$ if $g$ 
    is odd and $k=\frac{g+2}{2}$ if $g\geq 4$ is even. 
    Assume that the conjectured estimate in \cite[p. 351-352]{HMo} is 
    true. Then if $g>>0$, we have $\alpha> 8\alpha'$.
 \end{lem}
 \begin{proof} By Equation 
      \ref{facile}, by Equation \ref{facilefacile} and by Corollary 
      \ref{primadifferenza} we have
      $$
      \alpha-8\alpha'=(b-1)\left[2e-{\widetilde{N}}_{1}+
      \frac{{\widetilde{N}}_{3}}{9}\right]-{\widetilde{N}}.
      $$\noindent
      Since $e\geq 0$ and $b=2(g+k-1)>0$
it is 
sufficient to show that $(b-1)\left[ 
-{\widetilde{N}}_{1}+\frac{{\widetilde{N}}_{3}}{9}\right]-{\widetilde{N}}\geq 0$. 
In the case of maximal gonality up to the first factor 
by \cite[bottom of the page 351]{HM} if $k>>0$ it is conjectured that
${\widetilde{N}}_{3}\simeq (k-2){\widetilde{N}}_{1}$, 
${\widetilde{N}}\simeq (k)(k-1)\frac {{\widetilde{N}}_{1}} {2}$. Finally 
assume that $g=2n+1$. By Equation \ref{casodisparo} we have 
$b-1=6n+3$ then up to the first order $(b-1)\left[ 
-{\widetilde{N}}_{1}+\frac{{\widetilde{N}}_{3}}{9}\right]-{\widetilde{N}}\simeq 
{\widetilde{N}}_{1}\left[ 
(6n+3)\left[ -1+\frac{n}{9}\right]-(n+2)(n+1)\frac{1}{2}\right]=
{\widetilde{N}}_{1}\frac{n^{2}-43n-18}{6}>0$ if $n\geq 44$, that is 
$k\ge 46$. 
By the same argument if $g=2n$ we have that $\alpha\geq 8\alpha'$ if 
$3n^{2}-131n+20\geq 0$ that is if $n\geq 43$ and then $k\ge 44$.

\end{proof}

\begin{prop}\label{secondo} 
 Let $f\colon F\to Y$ be an 
    Harris Morrison genus-$g$ fibration 
   starting from any plane curve $X$ of 
       genus $g(X)>>0$. Assume 
    that the gonality $k$ is maximal and that the conjectured 
    estimates in \cite[p. 351-352]{HMo} are
    true. If $g$ is big enough, then $F$ is a surface of 
    positive index i.e. $K^{2}_{F}> 8\chi(\sO_{F})$. Moreover if $f\colon F\to Y$
    is general in its class then the irregularity of $F$ is
    $g(Y)$. In particular, $f\colon F\to Y$ is the Albanese morphism 
    of $F$.
%
%
\end{prop}
\begin{proof} For every Harris Morrison family as in the 
 statment:
      
\begin{equation}\label{nonfacilefacilefacile}
	  \frac{K^{2}_{F}}{\chi(\sO_{F})}=
	  \frac{c\alpha+8{\widetilde{N}}(g-1)(g(X)-1)}{c\alpha'+{\widetilde{N}}
	  (g-1)(g(X)-1)}
\end{equation}
\noindent since Equation 
      \ref{facile} and Equation \ref{facilefacile}. Then the first claim is 
      equivalent to show that $\alpha\geq 8\alpha'$ and under our 
      assumption this 
      follows by Lemma \ref{tecnico}. Let $q(F)$ be the irregularity 
      of $Y$. By contradiction assume that $q(Y)>g(Y)$. Then by 
      universal property of the Albanese morphism it follows that the 
      Jacobian of the general fiber of $f\colon F\to Y$ contains an Abelian 
      subvariety; but the locus of such curves in $\overline\sM_{g}$ 
      is a proper 
      closed. Hence the family is not a sweeping one: a contradiction 
      to Proposition \ref{harris}.
\end{proof}

\noindent We have shown the Proposition stated in the Introduction.

We conclude by observing that if the conjectured estimates of Harris Morrison 
hold, then the families $F$ of Proposition \ref{secondo} 
furnish an intriguing example of
surfaces 
with ratio $\frac{K^{2}_{F}}{\chi(\sO_{F})}$ 
asymptotically $8$, which are minimal as semistable models, but with slope 
of the supported fibration aymptotically equal to $12$. 
We think that this kind of 
divergence between the two fundamental ratios among
the invariants of a fibered surface which is minimal as a semistable model
is worthy to be studied in the light of the recent results of 
Urz{\'u}a quoted in the Introduction of this paper.

\end{document}